\documentclass[12pt]{amsart}
\usepackage{geometry}
\usepackage{graphicx}
\usepackage{amscd}
\usepackage{amssymb}
\usepackage{amsmath}
\usepackage[latin1]{inputenc}
\usepackage{epsfig,pstricks,pst-node,pst-tree,ifthen}
\usepackage{amsfonts}
\usepackage{bm}
\usepackage[all]{xy}
\usepackage{cite}
\usepackage[normalem]{ulem}

\usepackage{tikz}
\usetikzlibrary{arrows.meta}
 \usetikzlibrary{fit,calc}

 \usepackage{mathdots}

\theoremstyle{plain}                                       %
\newtheorem{thm}{\quad Theorem}                            %
\newtheorem{cor}[thm]{\quad Corollary}                     %
\newtheorem{prop}[thm]{\quad Proposition}                  %
\theoremstyle{definition}                                  %
\newtheorem{defi}[thm]{\quad Definition}                   %
\newtheorem{rmk}[thm]{\quad Remark}                        %
\newtheorem{ejem}[thm]{\quad Example}                      %

\newcommand{\N}{{\mathbb N}}
\newcommand{\K}{{\mathbb K}}

\newcommand{\Qm}{{\mathcal Q}}
\newcommand{\Cm}{{\mathcal C}}

\newcommand{\wt}{{\texttt{wt}}}
\newcommand{\al}{{\alpha}}
\newcommand{\be}{{\beta}}

\newcommand{\p}{{\partial}}

\newcommand{\om}{{\omega}}
\newcommand{\g}{{\gamma}}

\newcommand{\Ri}{{\mathcal{R}}}
\newcommand{\D}{{\mathcal D}}
\newcommand{\W}{{\mathcal W}}
\newcommand{\G}{{\mathcal G}}

\begin{document}

\vspace{1cm}

\title{Diagonals in Riordan Matrices and Applications}

\author{Gi-Sang Cheon}
\address{Department of Mathematics, Sungkyunkwan University, Suwon 16419, Rep. of Korea}
\email{gscheon@skku.edu}

\author{Ana Luz\'on}
\address{Departamento de Matem\'atica Aplicada, Universidad Polit\'ecnica de Madrid, Spain}
\email{anamaria.luzon@upm.es}

\author{Manuel A. Mor\'on}
\address{Departamento de Algebra, Geometr\'ia y Topolog\'ia,
Universidad Complutense de Madrid and Instituto de Matem\'atica
Interdisciplinar (IMI), Spain}
\email{mamoron@mat.ucm.es}

\author{Jos\'e L. Ram\'irez}
\address{Departamento de Matem\'aticas, Universidad Nacional de Colombia, Colombia}
\email{jlramirezr@unal.edu.co}

\begin{abstract}
We introduce a method for describing Riordan matrices via recurrence relations along their diagonals. This provides a new structural description that complements the classical row-wise and column-wise constructions via the A-sequence. As an application, we characterize families of palindromic Riordan polynomials, yielding a new combinatorial interpretation in terms of lattice paths.
\end{abstract}

\maketitle

\begin{center}
{\it In memory of Hanna Kim, Renzo Sprugnoli, and Emmanuel Munarini.}
\end{center}



\section{Introduction} 

Let $\K[[x]]$ be the ring of formal power series over a filed $\K$, and let $\mathbb N=\{0,1,\ldots\}$. A {\it Riordan matrix} \cite{Sha91}, denoted by $D=(d,h)$, is an infinite lower triangular matrix
$D=(d_{i,j})_{i,j\in\mathbb N}$ defined by $d_{i,j}=[x^i]dh^j$, where $d,h\in\K[[x]]$ satisfy $d(0)\neq0$, $h(0)=0$, and $h'(0)\neq0$. Here $[x^i]$ denotes the coefficient extraction operator. 

Riordan matrices can be constructed in various ways. In \cite{Rog78}, Rogers constructs each element of a certain subgroup of the Riordan group using what he calls the {\it A-sequence}. This provides a horizontal, or row-wise, construction of these matrices. It was later shown that this construction extends to arbitrary Riordan matrices, thereby establishing the A-sequence as a fundamental tool in the study of Riordan matrices. Alternatively, any Riordan matrix can be constructed vertically, or column-wise, using the A-sequence of its inverse; see \cite{2ways, teo}. Both constructions were later improved by including all powers of the A-sequence of the Riordan matrix and all powers of the A-sequence of its inverse; see \cite{Constr}, or, more clearly, in Theorem 1 on p. 240 of \cite{Double}.
Furthermore, in \cite{Cheon-Jin}, the authors present a family of constructions for Riordan matrices depending on the slopes of certain lines; these are called $(a,b)$-sequences. Among other related tools, one finds the so-called A-matrix introduced in \cite{Merlini}.

Riordan matrices are highly structured infinite lower triangular matrices. While their structure has traditionally been analyzed in terms of rows and columns, it is natural to investigate whether an alternative description can be formulated in terms of {\it diagonals}, interpreted as the lines parallel to the main diagonal. Addressing this question is the main objective of the present article. Although some previous results in this direction for finite Riordan matrices can be found in \cite{finitas}, p.~245, our main motivation comes from the rich diagonal patterns observed in many bi-infinite Riordan matrices, as discussed in Section~2.

Recently, the diagonals of certain Riordan matrices have been useful for constructing the derived series of the Riordan group: see \cite{commutador}. In Subsection \ref{Sub:Gk}, we give an alternative proof for the description of the groups $\mathcal{G}_k$, which are essential for this calculation.

Another motivation for our work comes from the study of palindromic phenomena in Riordan matrices.
It is natural to ask when a family of Riordan polynomials forms a palindromic sequence, as in Pascal's triangle. This question was posed and answered by Hana Kim in 2018. Her result is presented for the first time in this work (see Proposition \ref{teo:Hanna}). Moreover, a diagonal characterization of palindromic Riordan matrices becomes transparent in our setting. For this reason, we present two approaches to the same result, one of which is Kim's original proof. Related work, including the study of palindromic generalized Riordan matrices, as well as a combinatorial interpretation, can be found in \cite{p. Petrullo}. We also give a new combinatorial interpretation, different from the one obtained in \cite{p. Petrullo}.

More specifically, in Section 2, we present motivating examples, along with the necessary preliminaries. In Section 3, we state and prove the main result, describing Riordan matrices using a recurrence formula for their diagonals. We then provide general examples with some specific properties. In Section 4, we apply the diagonal description to characterize all families of palindromic Riordan polynomials, again with several examples. Finally, in Section 5, we conclude with a combinatorial interpretation of palindromic Riordan matrices in terms of lattice paths.


\section{Preliminaries and Motivation} 

\subsection{Some preliminaries results}
For notational reasons, we recall from  \cite{teo} that a Riordan matrix $D=(d,h)$ can be represented in terms of $g,f\in\K[[x]]$ with $f(0)\ne0$ and $g(0)\ne0$ as follows:
\[
T(f\mid g)=D=(d_{i,j})_{i,j\in\N},\quad d=\frac{f(x)}{g(x)}\;{\text {and}}\;h=\frac{x}{g(x)}.
\]
For instance, the Riordan matrix $T(1\mid 1-x)$ represents Pascal's triangle $\left(\frac{1}{1-x},\frac{x}{1-x}\right)$. 

The above definition can be reinterpreted as follows: the generating function of the $j$-th column (starting at $j=0$) of $D$ is the formal power series $x^jf(x)/g^{j+1}(x)$, which is well-defined since $g(0)\neq 0$. Hence, $D$ is a lower triangular matrix and invertible because $f(0)\neq0$.

 In \cite{teo}, it was defined an ultrametric space $(\K[[x]],d)$ where $\displaystyle{d(\al,\be)=\frac{1}{2^{\om(\al-\be)}}}$, being $\om(s)$ the order of the formal power series $s$. So, if we consider $T(f\mid g)$ as a continuous endomorphism in $(\K[[x]],d)$, Proposition 19 in \cite{teo} states the following:
\begin{equation}\label{TFRA}
    T(f \mid  g)(\gamma)=\frac{f(x)}{g(x)}\,\gamma\!\left(\frac{x}{g(x)}\right), \quad \forall \ \g\in\K[[x]].
\end{equation}
The above result is known as the {\it fundamental theorem of Riordan matrices}.

The Riordan group (that is,  the set of all Riordan matrices) is a subgroup of the group of invertible infinite lower triangular matrices, with the usual matrix product as the operation. The product is given by
\[T(f\mid g)T(l\mid m)=T\left(fl\left(\frac{x}{g}\right)\big|
gm\left(\frac{x}{g}\right)\right),\] where
$fl\left(x/g\right):= f(x)\cdot l\left(x/g(x)\right)$ and analogously for the second term. 

The inverse is given by
\[(T(f\mid g))^{-1}\equiv T^{-1}(f\mid g)=T\left(\frac{1}{f(\frac{x}{A})}\Big| A \right),\]
where
$\left(x/A\right)\circ\left(x/g\right)=\left(x/g\right)\circ\left(x/A\right)=x$. See \cite[Proposition 7, p. 3615]{2ways} for more details. The power series  $A$ is the so-called \emph{$A$-sequence} of $T(f\mid g)$. Clearly, the $A$-sequence of $T(f\mid g)$ depends only on the power series $g$. Moreover, if $A=\sum_{k\geq0}a_{k}x^{k}$, then
\[d_{i,j}=\sum_{k=0}^{i-j}a_kd_{i-1,j-1+k}, \qquad i,j\geq1.\]

\subsection{Some motivating examples} The bi-infinite representation of Riordan matrices has been studied by several authors; see, for instance, \cite{Barnabei, Barnabei2, Brini, Cheon-Jin, Lie, 2ways, Constr, Complem, finitas, Spr94, VerdeStar85, VerdeStar04}. 

Let $D=T(f\mid g)=(d_{i,j})_{i,j\in\mathbb N}$ be a Riordan matrix.
The associated bi-infinite matrix $(d_{n,k})_{n,k\in\mathbb Z}$
can be partitioned into four infinite submatrices as
\[
(d_{n,k})_{n,k\in\mathbb Z}=\left(
\begin{array}{c|c}
A & O\\
\hline
B & D
\end{array}
\right),
\]
where $D$ is the original Riordan matrix, and
$O$ denotes the infinite zero matrix. The block $A$ is determined as either the \emph{dual Riordan array}, denoted by $D^{\diamond}$,
or the \emph{complementary Riordan array}, denoted by $D^{\bot}$, associated with the Riordan matrix $D$. The precise form of $A$ depends on the choice of coordinate axes for the bi-infinite matrix. See \cite{Complem, finitas} for more information about $D^{\diamond}$ and $D^{\bot}$. 

In general, we do not know how to describe the block $B$, although in many examples it exhibits recognizable patterns.  In connection with this goal, we obtain partial results for some matrices satisfying  $D=D^{\diamond}$; however, further work is required to achieve a complete description.

For example, in the case of Pascal's triangle  $D=T(1\mid 1-x)$, we have:

\[
 \begin{smallmatrix}
\left(
\begin{array}{ccccccc|ccccccc}
 \ddots & \vdots & \vdots & \vdots & \vdots & \vdots & \vdots
        & \vdots & \vdots & \vdots & \vdots & \vdots & \vdots & \iddots\\
 \cdots & 1  & 0 & 0 & 0 & 0 & 0 & 0 & 0 & 0 & 0 & 0 & 0 & \cdots\\
 \cdots & -5 & 1 & 0 & 0 & 0 & 0 & 0 & 0 & 0 & 0 & 0 & 0 & \cdots\\
 \cdots & 10 & -4 & 1 & 0 & 0 & 0 & 0 & 0 & 0 & 0 & 0 & 0 & \cdots\\
 \cdots & -10 & 6 & -3 & 1 & 0 & 0 & 0 & 0 & 0 & 0 & 0 & 0 & \cdots\\
 \cdots & 5 & -4 & 3 & -2 & 1 & 0 & 0 & 0 & 0 & 0 & 0 & 0 & \cdots\\
 \cdots & -1 & 1 & -1 & 1 & -1 & 1 & 0 & 0 & 0 & 0 & 0 & 0 & \cdots\\
  \hline
 \cdots & 0 & 0 & 0 & 0 & 0 & 0 & 1 & 0 & 0 & 0 & 0 & 0 & \cdots\\
 \cdots & 0 & 0 & 0 & 0 & 0 & 0 & 1 & 1 & 0 & 0 & 0 & 0 & \cdots\\
 \cdots & 0 & 0 & 0 & 0 & 0 & 0 & 1 & 2 & 1 & 0 & 0 & 0 & \cdots\\
 \cdots & 0 & 0 & 0 & 0 & 0 & 0 & 1 & 3 & 3 & 1 & 0 & 0 & \cdots\\
 \cdots & 0 & 0 & 0 & 0 & 0 & 0 & 1 & 4 & 6 & 4 & 1 & 0 & \cdots\\
 \cdots & 0 & 0 & 0 & 0 & 0 & 0 & 1 & 5 & 10 & 10 & 5 & 1 & \cdots\\
 \iddots & \vdots & \vdots & \vdots & \vdots & \vdots & \vdots
        & \vdots & \vdots & \vdots & \vdots & \vdots & \vdots & \ddots\\
\end{array}
\right).
 \end{smallmatrix}
\]
Here $B=O$, and $A$ corresponds to $D^{\bot}$, which in this setting coincides with its inverse.

Motivated by the above problem, we began to study the diagonals recursively. Anyone who has worked with Riordan matrices has likely noticed patterns in their diagonals. Furthermore, in the study of the derived series of the Riordan group in \cite{commutador}, we observe that the diagonals, at least the first diagonals,  play an important role. We return to this point in Subsection \ref{Sub:Gk}. All of this led us to consider describing Riordan matrices by their diagonals, in the same spirit as the classical descriptions by rows or by columns. Next, we present several Riordan matrices and examine their diagonals.
 
  In \cite{BanPas}, there appears a particular Riordan matrix of special interest in our setting, and \cite{2ways} contains its bi-infinite representation: 
\[
\left(
\begin{array}{cccccccccc}
\color{red}{-1} & 0 & 0 & 0 & 0 & 0 & 0 & 0 & 0 & \color{brown}{0} \\
\color{green}{8} & \color{red}{1} & 0 & 0 & 0 & 0 & 0 & 0 & \color{brown}{0} & 0 \\
\color{blue}{-23} & \color{green}{-6} & \color{red}{-1} & 0 & 0 & 0 & 0 & \color{brown}{0} & 0 & 0 \\
26 & \color{blue}{11} & \color{green}{4} & \color{red}{1} & 0 & 0 & \color{brown}{0} & 0 & 0 & 0 \\
-5 & -4 & \color{blue}{-3} & \color{green}{-2} & \color{red}{-1} & \color{brown}{0} & 0 & 0 & 0 & 0 \\
-4 & -3 & -2 & \color{blue}{-1} & \color{brown}{0} & \color{red}{1} & 0 & 0 & 0 & 0 \\
-3 & -2 & -1 & \color{brown}{0} & \color{blue}{1} & \color{green}{2} & \color{red}{-1} & 0 & 0 & 0 \\
-2 & -1 & \color{brown}{0} & 1 & 2 & \color{blue}{3} & \color{green}{-4} & \color{red}{1} & 0 & 0 \\
-1 & \color{brown}{0} & 1 & 2 & 3 & 4 & \color{blue}{-11} & \color{green}{6} & \color{red}{-1} & 0 \\
\color{brown}{0} & 1 & 2 & 3 & 4 & 5 & -26 & \color{blue}{23} & \color{green}{-8} & \color{red}{1} \\
\end{array}
\right)
\]

In this matrix, we observe that the bi-infinite diagonals are antisymmetric with respect to the diagonal perpendicular to the main diagonal, shown in brown and consisting of zeros.

After Pascal's triangle, perhaps the most famous numerical triangle is Catalan's triangle. The following matrix is related to the Catalan triangle in its bi-infinite form. One can observe certain symmetries and antisymmetries in its diagonals:
\[
\left(
\begin{array}{cccccccccccc}
\color{red}{1} &  &  &  &  &  &  &  &  &   &  \\
\color{green}{-5} & \color{red}{1} &  &  &  &  &  &  &  &  \\
\color{blue}{5} & \color{green}{-4} & \color{red}{1}  &  &  &  &  &  &  &  \\
0 & \color{blue}{2} & \color{green}{-3} & \color{red}{1}   &  &  &  &  &  &  \\
0 & 0 & \color{blue}{0} & \color{green}{-2} & \color{red}{1}   &  &  &  &  &  \\
-1 & -1 & -1 & \color{blue}{-1} & \color{green}{-1} & \color{red}{1} &    &  &  &  \\
-5 & -4 & -3 & -2 & \color{blue}{-1} & \color{green}{0} & \color{red}{1}   &  &  &  \\
-20 & -14 & -9 & -5 & -2 & \color{blue}{0} & \color{green}{1} & \color{red}{1}  &  &  \\
-75 & -48 & -28 & -14 & -5 & 0 & \color{blue}{2} & \color{green}{2} & \color{red}{1}   &  \\
-275 & -165 & -90 & -42 & -14 & 0 & 5 & \color{blue}{5} & \color{green}{3} & \color{red}{1}   \\
\end{array}
\right)
\]

In \cite{finitas}, the following bi-infinite matrix appears on p.  261 as Example 38:
\[
\left(
\begin{array}{cccc|ccccc}
\color{red}{-1} &  &  &  &  &  &  &  &  \\
\color{green}{7} & \color{red}{1} &  &  &  &  &  &  &  \\
\color{blue}{-35/2} & \color{green}{-5} & \color{red}{-1} &  &  &  &  &  &  \\
35/2 & \color{blue}{15/2} & \color{green}{3} & \color{red}{1} &  &  &  &  &  \\
-35/8 & -5/2 & \color{blue}{-3/2} & \color{green}{-1} & \color{red}{-1} &  &  &  &  \\
\hline
-7/8 & -5/8 & -1/2 & \color{blue}{-1/2} & \color{green}{-1} & \color{red}{1} &  &  &  \\
-7/16 & -3/8 & -3/8 & -1/2 & \color{blue}{-3/2} & \color{green}{3} & \color{red}{-1} &  &  \\
-5/16 & -5/16 & -3/8 & -5/8 & -5/2 & \color{blue}{15/2} & \color{green}{-5} & \color{red}{1} &  \\
-35/128 & -5/16 & -7/16 & -7/8 & -35/8 & 35/2 & \color{blue}{-35/2} & \color{green}{7} & \color{red}{-1} \\
\end{array}
\right)
\]
This matrix is a Riordan self-dual involution, that is, $D=D^{\diamond}=D^{-1}$. It is highly symmetric in several ways. In particular, it contains palindromic sequences of numbers and palindromic sequences of polynomials.



\section{Diagonals: a recurrence formula and examples} 

\subsection{A recurrence formula for diagonals of Riordan matrices} 

Let $D=T(f\mid g)=(d_{i,j})_{i,j\in\N}$ be a Riordan matrix, where  $f(x)=\sum_{n\geq0}f_nx^n$ and $g(x)=\sum_{n\geq0}g_nx^n$, with $f_0\neq0$ and $g_0\neq0$. As shown in  \cite{teo} and more explicitly in \cite{2ways} (p.~3611), the entries $d_{i,j}$ satisfy the following recurrences. 
First,
\begin{equation}\label{ecu:d}
  d_{0,0}=\frac{f_0}{g_0}, \quad d_{i,0}=-\frac{g_1}{g_0}d_{i-1,0}-\frac{g_2}{g_0}d_{i-2,0}-\cdots-\frac{g_i}{g_0}d_{0,0}+\frac{f_i}{g_0}, \quad i\geq1.
\end{equation}
If $j>0$, then
\begin{equation}\label{ecu:dij}
d_{i,j}=-\frac{g_1}{g_0}d_{i-1,j}-\frac{g_2}{g_0}d_{i-2,j}-\cdots-\frac{g_{i-j}}{g_0}d_{j,j}+\frac{d_{i-1,j-1}}{g_0}, \quad i\geq1.
\end{equation}
As a consequence, the row generating polynomials $p_n(x)=\sum_{k=0}^{n} d_{n,k}\,x^k$ satisfy (see Theorem~5 in \cite{poly}) 
\begin{equation}\label{ecu:poly}
p_n(x)=\left(\frac{x-g_1}{g_0}\right)p_{n-1}(x)-\frac{g_2}{g_0}p_{n-2}(x)-\cdots-\frac{g_{n}}{g_0}p_{0}(x)+\frac{f_{n}}{g_0},
\qquad n\geq1,
\end{equation}
with initial value $p_0(x)=d_{0,0}=f_0/g_0$.

The recurrences \eqref{ecu:d} and \eqref{ecu:dij} describe $D$ column-wise, while \eqref{ecu:poly} gives a row-wise description. We now introduce an alternative description in terms of diagonals. For $n\ge 0$, define the generating function of the {\it $n$th diagonal of a Riordan matrix} $D$ by
\[
\Delta_n(x)=\sum_{k\ge 0} d_{k+n,k}\,x^k,
\]
and consider the associated bivariate generating function
\[
\Delta(z,x)=\sum_{n\ge 0} \Delta_n(x)\,z^n.
\]

\begin{thm}{\label{T:Diag}}
For all $n\geq 0$, 
\[
\Delta_n(x)=\frac{1}{g_0-x}\left(f_n-\sum_{l=1}^{n}g_l\Delta_{n-l}(x)\right).
\]
Moreover, 
\[\Delta(z,x)=\frac{f(z)}{g(z)-x}.\]
\end{thm}
\begin{proof}
For $n=0$ we have
  \[
  \Delta_0(x)=\sum_{k\geq0}d_{k,k}x^k=\sum_{k\geq0}\frac{f_0}{g_0^{k+1}}x^k=
  \frac{f_0}{g_0}\sum_{k\geq0}\left(\frac{x}{g_0}\right)^k=\frac{f_0}{g_0-x},
  \]
which agrees with the stated formula (with the convention that $\sum_{\ell=1}^{0}(\cdots)=0$).

Now let  $n\geq1$. Using $\eqref{ecu:d}$ for the zeroth column and $\eqref{ecu:dij}$ for the remaining columns, we have 
\begin{align*}
\Delta_n(x)
&=\sum_{k\ge 0} d_{k+n,k}\,x^k =d_{n,0}+\sum_{k\ge 1} d_{k+n,k}\,x^k\\
&=\left(-\frac{1}{g_0}\sum_{\ell=1}^{n} g_\ell\,d_{n-\ell,0}+\frac{f_n}{g_0}\right)
+\sum_{k\ge 1}\left(-\frac{1}{g_0}\sum_{\ell=1}^{n} g_\ell\,d_{k+n-\ell,k}
+\frac{d_{k+n-1,k-1}}{g_0}\right)x^k\\
&=-\frac{1}{g_0}\sum_{\ell=1}^{n} g_\ell\sum_{k\ge 0} d_{k+n-\ell,k}\,x^k
+\frac{x}{g_0}\sum_{k\ge 1} d_{k+n-1,k-1}\,x^{k-1}
+\frac{f_n}{g_0}\\
&=-\frac{1}{g_0}\sum_{\ell=1}^{n} g_\ell\,\Delta_{n-\ell}(x)
+\frac{x}{g_0}\Delta_n(x)
+\frac{f_n}{g_0}.
\end{align*}
Collecting the $\Delta_n(x)$ terms gives
\[
\Delta_n(x)\left(1-\frac{x}{g_0}\right)
=\frac{f_n}{g_0}-\frac{1}{g_0}\sum_{\ell=1}^{n} g_\ell\,\Delta_{n-\ell}(x),
\]
and therefore
\[
\Delta_n(x)=\frac{1}{g_0-x}\left(f_n-\sum_{\ell=1}^{n} g_\ell\,\Delta_{n-\ell}(x)\right).
\]
Finally, define $f(z)=\sum_{n\ge 0} f_n z^n$ and $g(z)=\sum_{n\ge 0} g_n z^n$. Multiply the identity
\[
(g_0-x)\Delta_n(x)=f_n-\sum_{\ell=1}^{n} g_\ell\,\Delta_{n-\ell}(x)
\]
by $z^n$ and sum over $n\ge 0$ to obtain
\[
(g_0-x)\Delta(z,x)
=f(z)-\sum_{n\ge 0}\sum_{\ell=1}^{n} g_\ell\,\Delta_{n-\ell}(x)\,z^n.
\]
Reindexing with $m=n-\ell\ge 0$ yields
\[
\sum_{n\ge 0}\sum_{\ell=1}^{n} g_\ell\,\Delta_{n-\ell}(x)\,z^n
=\sum_{\ell\ge 1} g_\ell z^\ell \sum_{m\ge 0}\Delta_m(x)z^m
=(g(z)-g_0)\Delta(z,x).
\]
Hence $
(g_0-x)\Delta(z,x)=f(z)-(g(z)-g_0)\Delta(z,x)$, 
and solving for $\Delta(z,x)$ gives
\[
\Delta(z,x)=\frac{f(z)}{g(z)-x}. \qedhere
\]
\end{proof}

%

\begin{rmk} In \cite{Spr94},  Sprugnoli introduced  what he called {\it the bivariate generating function associated with $D$} for any Riordan array $D$. If $D=T(f|g)$, this function is  $\frac{f(z)}{g(z)-xz}$. Note that $\frac{f(z)}{g(z)-xz}=\sum_{n\geq0}z^n\Delta_n(xz)$, which is not the same as $\Delta(z,x)$ in Theorem \ref{T:Diag}.
\end{rmk} 

\subsection{Diagonals of the Pascal's  triangle.}\label{SS:PascDiag} Consider Pascal's triangle, namely the Riordan matrix  $P=T(1\mid 1-x)$. Thus $f(x)=1$ and $g(x)=1-x$, so $f_0=1$ and $f_n=0$ for all $n\geq1$, while $g_0=1$, $g_1=-1$, and $g_n=0$ for all $n\geq2$. By Theorem \ref{T:Diag} 
\[
\Delta_0(x)=\frac{f_0}{g_0-x}=\frac{1}{1-x},
\]
and for $n\geq1$,
\begin{align*}
\Delta_n(x)&=\frac{1}{g_0-x}\left(f_n-\sum_{\ell=1}^{n}g_\ell\Delta_{n-\ell}(x)\right)\\
&=\frac{1}{1-x}\left(0-\sum_{\ell=1}^{1}g_\ell\Delta_{n-\ell}(x)\right)\\
&=\frac{1}{1-x}\Delta_{n-1}(x).
\end{align*}
Iterating this recurrence gives
\[
\Delta_n(x)=\frac{1}{(1-x)^{n+1}}=\sum_{k\geq 0}\binom{n+k}{k}x^k, \quad n\geq 0.
\]
Moreover, 
\[
\Delta(z,x)=\frac{f(z)}{g(z)-x}=\frac{1}{1-x-z}.
\]
Note that, in this case, the generating function of the $n$th column of Pascal's triangle is $x^{n}\Delta_n(x)$.

\subsection{Diagonals of Toeplitz matrices.}

Consider the Toeplitz matrix $T(f\mid 1)$, where $f(x)=\sum_{n\geq 0} f_n x^n$. In this case  $g(x)=1$, so $g_0=1$ and $g_n=0$ for all $n\geq1$. Hence,  by Theorem \ref{T:Diag}, for every $n\geq 0$,
\[
\Delta_n(x)=\frac{1}{g_0-x}\left(f_n-\sum_{\ell=1}^{n}g_\ell\Delta_{n-\ell}(x)\right)=\frac{f_n}{1-x}.
\]
Moreover,
\[
\Delta(z,x)=\frac{f(z)}{1-x}.
\]
The first rows of $T(f\mid 1)$ are
\[\left(
  \begin{array}{cccccc}
    f_0 &  &  &  &  \\
    f_1 & f_0 &  &  &O \\
    f_2 & f_1 & f_0 &  &  \\
    f_3 & f_2 & f_1 & f_0 &  \\
    f_4 & f_3 & f_2 & f_1 & f_0 \\
    \vdots &  & \vdots &  & \vdots&\ddots
  \end{array}
\right).
\]

\subsection{Diagonals of the Catalan triangle}
The Catalan Riordan array is classically given by $(C(z),zC(z))$, where
\[
C(z)=\sum_{n\ge 0} C_n z^n=\frac{1-\sqrt{1-4z}}{2z}
\]
is the generating function of the Catalan numbers $C_n=\frac{1}{n+1}\binom{2n}{n}$. In our notation $T(f\mid g)$, this array corresponds to
\[
T\!\left(1\ \middle|\ \frac{1}{C(z)}\right)
=
T\!\left(1\ \middle|\ \frac{1+\sqrt{1-4z}}{2}\right).
\]
Its first few rows are
\[
\left(
\begin{array}{cccccc}
1 \\
1&1&&&O\\
2&2&1\\
5&5&3&1\\
14&14&9&4&1\\
    \vdots &  & \vdots &  & \vdots&\ddots
\end{array}
\right).
\]
To obtain a recurrence for the diagonal generating functions, write $g(z)=\sum_{\ell\ge 0} g_\ell z^\ell$ and note that
\[
g(z)=\frac{1}{C(z)}=\frac{1+\sqrt{1-4z}}{2}
=1-\sum_{\ell\ge 1} C_{\ell-1}\,z^\ell.
\]
Hence $g_0=1$ and $g_\ell=-C_{\ell-1}$ for $\ell\ge 1$, while $f_0=1$ and $f_n=0$ for $n\ge 1$. Therefore, Theorem~\ref{T:Diag} yields
\[
\Delta_0(x)=\frac{1}{1-x},
\qquad
\Delta_n(x)=\frac{1}{1-x}\sum_{\ell=1}^{n} C_{\ell-1}\,\Delta_{n-\ell}(x),
\quad n\geq 1.
\]
Finally, extracting coefficients from $\Delta(z,x)=\frac{C(z)}{1-xC(z)}$ gives, for every $n\ge 0$,
\[
\Delta_n(x)=[z^n]\Delta(z,x)=\sum_{k\ge 0} x^k\,[z^n]\,C(z)^{k+1}.
\]
Using the classical coefficient formula (cf. \cite{Wilf})
\[
[z^n]\,C(z)^{k+1}=\frac{k+1}{2n+k+1}\binom{2n+k+1}{n}
\qquad (k,n\ge 0),
\]
we obtain the explicit expression
\[
\Delta_n(x)=\sum_{k\ge 0}\frac{k+1}{2n+k+1}\binom{2n+k+1}{n}\,x^k.
\]

\subsection{Diagonals of the elements of the groups $\G_k$} \label{Sub:Gk}
Recall the definition of the groups $\G_k$. For $k\geq 2$, set \[\G_k=\{T(g\mid g): g\in(1+ x^{k-1}\mathbb K[[x]])\},\] or equivalently, $\G_k=\{(1,h): h\in (x+ x^k\mathbb K[[x]])\}$. See \cite{Jennings,teo}.  These groups play a basic role in the description of the derived series of the Riordan group. In particular, one has
\[
\Ri^{(n)}=\Bigl\{\,T(f\mid g)\in\Ri:\ f\in\bigl(1+z^{2^n-n}\mathbb K[[z]]\bigr),\ T(g\mid g)\in \G_{2^{n}-1}\Bigr\},
\]
or equivalently 
\[
\Ri^{(n)}=\Bigl\{\, (d,h)\in\mathcal R:\ d\in\bigl(1+z^{2^n-n}\mathbb K[[z]]\bigr),\ (1,h)\in \G_{2^{n}}\Bigr\}.
\]
See \cite{commutador}. In particular, the derived series of the associated subgroup satisfies $\mathcal A^{(n)}=\mathcal G_{2^n}$.

The next result provides a direct proof of part~(2) of Theorem~4 in \cite{commutador}, using the diagonal formula from Theorem~\ref{T:Diag}.

\begin{thm}
  The first few diagonals of the elements in $\G_k$ are 
  \[
d_{j+m,j}=  \left\{
\begin{array}{ll}
    1, & \hbox{$ m=0$;} \\
    0, & \hbox{$1\leq m\leq k-2$;} \\
    -jg_m, & \hbox{$k-1\leq m\leq 2k-3$.}
  \end{array}
\right.
\]
\end{thm}

\begin{proof}
Since $T(g\mid g)\in \G_k$, we have $g_0=1$ and $g_1=\cdots=g_{k-2}=0$.
Moreover, here $f=g$, so Theorem~\ref{T:Diag} gives
\[
\Delta_m(x)=\frac{1}{1-x}\left(g_m-\sum_{\ell=1}^{m} g_\ell\,\Delta_{m-\ell}(x)\right),
\qquad m\ge 0.
\]
In particular,
 \[
\Delta_0=\frac{f_0}{g_0-x}=\frac{1}{1-x}.
\]
If $1\le m\le k-2$, then $g_m=0$ and also $g_\ell=0$ for $1\le \ell\le m$, hence $\Delta_m(x)=0$.


Now assume $k-1\le m\le 2k-3$. In the sum
$\sum_{\ell=1}^{m} g_\ell\,\Delta_{m-\ell}(x)$,
the terms with $1\le \ell\le k-2$ vanish because $g_\ell=0$, while the terms with
$k-1\le \ell\le m-1$ vanish because $1\le m-\ell\le k-2$ and thus $\Delta_{m-\ell}(x)=0$. Therefore only the term $\ell=m$ remains, and we obtain
\[
\Delta_m(x)=\frac{1}{1-x}\bigl(g_m-g_m\Delta_0(x)\bigr)
=\frac{g_m}{1-x}\left(1-\frac{1}{1-x}\right)
=-\frac{g_m\,x}{(1-x)^2}.
\]
Since $\Delta_m(x)=\sum_{j\ge 0} d_{j+m,j}x^j$ and
\[
\frac{x}{(1-x)^2}=\sum_{j\ge 1} j\,x^j,
\]
it follows that $d_{j+m,j}=-j g_m$ for $k-1\le m\le 2k-3$ (and $d_{m,0}=0$), as claimed.
\end{proof}

\begin{cor}[Part (2) of Theorem~4 in \cite{commutador}]
Let $(1,h)=(d_{i,j})_{i,j\ge 0}\in \mathcal G_k$, with
\[
h(z)=z+h_k z^k+h_{k+1} z^{k+1}+\cdots.
\]
Then, for all $j\ge 0$,
\[
d_{j+m,j}=
\begin{cases}
1, & m=0,\\[2pt]
0, & 1\le m\le k-2 \quad (k\ge 3),\\[2pt]
j\,h_{m+1}, & k-1\le m\le 2k-3.
\end{cases}
\]
\end{cor}

\begin{proof}
If $T(g\mid g)=(1,h)\in\G_k$, then  $g=x/h$,  or equivalently $h=x/g$. Therefore, 
\[
\left\{
  \begin{array}{ll}
    g_0=h_1=1,\\
    g_m=h_{m+1}=0, & 1\le m\le k-2,\\
    g_m=-h_{m+1}, & k-1\le m\le 2k-3.
  \end{array}
\right.
\]
This completes the proof.
\end{proof}

\subsection{Diagonals of two particularly related triangles.}
Consider the following two triangles, which originally motivated our study of Riordan matrices. They were introduced in \cite{BanPas} and have since been further investigated by Luz\'on and Mor\'on.

Let
\[
D=T\left(\frac{1}{(1-x)^2}\Big| 2x-1\right)=\begin{pmatrix}-1& && &&& & \\ -4&1&& &O&&&\\ -11&6&-1&&&&&\\
-26&23&-8&1&&&&\\ -57&72&-39&10&-1&&&\\
-120&201&-150&59&-12&1&&\\
\vdots&\vdots&\vdots&\vdots&\vdots&\vdots&\ddots&\\
\end{pmatrix}.
\]
Here $f(x)=\frac{1}{(1-x)^2}=\sum_{n\ge 0}(n+1)x^n$, so $f_n=n+1$ for all $n\ge 0$. Moreover, $g(x)=2x-1$, so $g_0=-1$, $g_1=2$, and $g_n=0$ for all $n\ge 2$. By Theorem~\ref{T:Diag},
\[
\Delta_0(x)=\frac{f_0}{g_0-x}=-\frac{1}{1+x},
\]
and for $n\ge 1$,
\[
\Delta_n(x)=\frac{1}{g_0-x}\bigl(f_n-g_1\Delta_{n-1}(x)\bigr)
=-\frac{1}{1+x}\bigl((n+1)-2\Delta_{n-1}(x)\bigr).
\]
Iterating the recurrence yields
\[
\Delta_n(x)=-\sum_{k=0}^{n}\frac{2^k\,(n+1-k)}{(1+x)^{k+1}}=\sum_{j\ge 0}\left((-1)^{j+1}\sum_{k=0}^{n}2^k\,(n+1-k)\binom{k+j}{j}\right)x^j,
\qquad n\ge 0.
\]

Moreover, the bivariate generating function is
\[
\Delta(z,x)=\frac{f(z)}{g(z)-x}
=-\frac{1}{(1-z)^2(1+x-2z)}.
\]

The following triangle is obtained from the previous one by adding the corresponding column as the first column,
that is, $\Phi^{-1}(D)=\widetilde D$ where $\Phi^{-1}(T(f|g))=T(fg|g)$. See \cite{Constr}.
\[
\widetilde D=T\left(\frac{2x-1}{(1-x)^2}\Big| 2x-1\right)=\begin{pmatrix}
1\\2&-1&&&&O\\3&-4&1\\4&-11&6&-1\\5&-26&23&-8&1\\6&-57&72&-39&10&-1\\
7&-120&201&-150&59&-12&1\\
\vdots&\vdots&\vdots&\vdots&\vdots&\vdots&\vdots&\ddots\\
\end{pmatrix}.
\]
Analogously, $f_n=n-1$ for all $n\geq0$, $g_0=-1$, $g_1=2$, and $g_n=0$ for all $n\geq2$. By Theorem~\ref{T:Diag},
\[
\widetilde\Delta_0(x)=\frac{f_0}{g_0-x}=\frac{1}{1+x},
\]
and for $n\ge 1$,
\[
\widetilde\Delta_n(x)
=\frac{1}{g_0-x}\bigl(f_n-g_1\widetilde\Delta_{n-1}(x)\bigr)
=-\frac{1}{1+x}\bigl((n-1)-2\widetilde\Delta_{n-1}(x)\bigr).
\]
Moreover,
\[
\widetilde\Delta(z,x)=\frac{\widetilde f(z)}{g(z)-x}
=\frac{\frac{2z-1}{(1-z)^2}}{(2z-1)-x}
=\frac{2z-1}{(1-z)^2(2z-1-x)}.
\]
Consequently, 
\[
\widetilde\Delta_n(x)=-\sum_{k=0}^{n}\frac{2^k\,(n-k-1)}{(1+x)^{k+1}},
\qquad n\ge 0.
\]
From the explicit expressions above, one can check that
\begin{align*}
\widetilde\Delta_n(x)&=-\sum_{k=0}^{n}\frac{2^k(n-k-1)}{(1+x)^{k+1}}=n+1-\sum_{k=0}^{n}2^k(n+1-k)\left(\frac{1}{(1+x)^{k}}-\frac{1}{(1+x)^{k+1}}\right)\\
    &=n+1-\sum_{k=0}^{n}2^k(n+1-k)\frac{x}{(1+x)^{k+1}}=f_n+x\Delta_n.
\end{align*}

This is an instance of a more general relation. In \cite{poly}, Proposition~15 describes the relation between the polynomial sequences associated with $T(f\mid g)$ and $T(fg\mid g)$ (see also \cite{poly}, Example~17 for the Morgan--Voyce families). For diagonals, we have the following direct analogue.

\begin{prop}\label{P:TfggvsTfg}
  Let $\Delta_n(x)$ be the diagonal generating functions  of $T(f\mid g)$, and  let $\widetilde\Delta_n(x)$ be the  diagonal generating functions of $T(fg\mid g)$.  Then, for every $n\ge 0$,
\[
\widetilde\Delta_n(x)=f_n+x\Delta_n(x),
\]
or equivalently,
\[
\Delta_n(x)=\frac{\widetilde\Delta_n(x)-f_n}{x}.
\]
\end{prop}
 This identity follows directly from the definition of $\Delta_n(x)$.

\subsection{Diagonals of $q$-cones.}
q-cones are simplicial complexes constructed iteratively by joining $[m]$ and $[q]$, denoted by $[m]\ast[q]$, where $[m]$ and $[q]$ are discrete simplicial complexes with $m$ and $q$ points, respectively. That is,
\[
C^{(0)}_{m,q}=[m], \qquad C^{(1)}_{m,q}=[m]\ast[q]=C^{(0)}_{m,q}\ast[q]
\]
and, iteratively $C^{(n+1)}_{m,q}=C^{(n)}_{m,q}\ast[q]$.

Let $F_{m,q}$ be the infinite matrix whose $n$th row is the $f$-vector of $C^{(n)}_{m,q}$, and let $\bar{F}_{m,q}$ be the corresponding infinite matrix for the extended $f$-polynomials; see \cite{q-conos}.

These matrices are
\[
F_{m,q}=T\left(\frac{m+q-mx}{q(1-x)}\Big|\frac{1-x}{q}\right), \qquad 
\bar{F}_{m,q}=T\left(\frac{m+(q-m)x}{q^2}\Big|\frac{1-x}{q}\right).
\]

For $F_{m,q}$ we have $f_0=\frac{m}{q}$ and $f_n=1$ for all $n\geq1$,  while $g_0=\frac{1}{q}$, $g_1=-\frac{1}{q}$, and $g_n=0$ for all $n\geq2$. Using the diagonal recurrence, we obtain
\[
\Delta_0(x)=\frac{m}{1-qx}, \qquad \Delta_1(x)=\frac{q}{1-qx}\left(1+\frac{1}{q}\Delta_0\right)=\frac{q}{1-qx}+\frac{\Delta_0}{1-qx},
\]
\[
\Delta_2(x)=\frac{q}{1-qx}\left(1+\frac{1}{q}\Delta_1\right)=\frac{q}{1-qx}+\frac{q}{(1-qx)^2}+\frac{\Delta_0}{(1-qx)^2},
\]
and hence, by induction,
\[
\Delta_n(x)=\frac{\Delta_0}{(1-qx)^{n}}+q\sum_{k=1}^{n}\frac{1}{(1-qx)^k}.
\]

Note that if $F_{m,q}=T(f\mid g)$, then $\bar{F}_{m,q}=T(fg\mid g)$. Let $\bar{\Delta}_n(x)$ denote the $n$-th diagonal of $\bar{F}_{m,q}$.  Using Proposition $\ref{P:TfggvsTfg}$, we obtain the relation between these two families of diagonals:
\[
\bar{\Delta}_0(x)=\frac{m}{q(1-qx)}\qquad \bar{\Delta}_n(x)=1+\frac{x\Delta_0(x)}{(1-qx)^{n}}+\sum_{k=1}^{n}\frac{qx}{(1-qx)^k}.
\]
Moreover, the bivariate generating functions are, respectively
\[
\Delta_{F_{m,q}}(z,x)
=\frac{m+(q-m)z}{(1-z)\,(1-z-qx)}, \qquad \Delta_{\bar F_{m,q}}(z,x)
=\frac{m+(q-m)z}{q\,(1-z-qx)}.
\]

\begin{rmk}
  The simplicial complexes $C^{(n)}_{m,q}$ are pure, in the sense that the number of facets coincides with the number of faces of maximal dimension. Thus, the $n$-th coefficient of $\Delta_0(x)$ is the number of facets of $C^{(n)}_{m,q}$. More generally, the coefficients on the diagonal $\Delta_k(x)$ can be interpreted as the number of faces of dimension $k$ less than the maximal one in each iteration $C^{(n)}_{m,q}$.
\end{rmk}

\section{Palindromic Riordan polynomials and a combinatorial interpretation}

In the previous section, we saw that the palindromic polynomials arising from the third quadrant
need not form a Riordan family. This naturally leads to the following question: for which Riordan
matrices are the row polynomials palindromic?

 \begin{defi}  Let $D=(d_{n,k})_{n,k\in\N}$ be a Riordan matrix. For each $n\geq 0$ define the \emph{row polynomial}
$p_n(x)=\sum_{k=0}^{n} d_{n,k}x^k$. We say that $D$ is a \emph{palindromic Riordan matrix} if each $p_n(x)$ is palindromic, that is, $p_n(x)=x^n p_n(1/x)$, or equivalently, 
$d_{n,k}=d_{n,n-k}$  for all $ n$ and  $0\leq k \leq n$.
 \end{defi}
 
The following result has an easy proof, but is important for our development, since it characterizes
palindromic Riordan matrices in terms of diagonals.
 
 \begin{prop}\label{prop:pal-columns-diagonals}
  Let $D=T(f\mid g)$ be a Riordan matrix. For $n\geq 0$, let $C_n(x)=\frac{x^nf(x)}{g^{n+1}(x)}$ be the generating function of the $n$th column of $D$. Then  $D$ is a palindromic Riordan matrix if and only if $C_n(x)=x^n\Delta_n(x)$ for all  $n\geq 0$.
 \end{prop}
 
\begin{proof}
By definition, 
\[
x^n\Delta_n(x)=\sum_{k\geq 0} d_{n+k,k}x^{n+k}.
\]
On the other hand, the generating function of the $n$th column of $D$ is
\[
C_n(x)=\frac{x^n f(x)}{g(x)^{n+1}}=\sum_{m\geq n} d_{m,n}\,x^m
=\sum_{k\geq 0} d_{n+k,n}\,x^{n+k}.
\]
Therefore $C_n(x)=x^n\Delta_n(x)$ for all $n\geq 0$ if and only if $d_{n+k,n}=d_{n+k,k}$ for all $n,k\geq 0$,
equivalently $d_{m,n}=d_{m,m-n}$ for all $m\geq 0$ and $0\leq n\leq m$. 
\end{proof}

Note that the above property is satisfied by Pascal's triangle; see Example~\ref{SS:PascDiag}.

\begin{prop}
 Let $D=T(f\mid g)=(d_{n,k})_{n,k\geq 0}$ be a Riordan matrix. Then $D$ is a palindromic Riordan  matrix  if and only if
\begin{equation*}
\Delta(z,xz)=\Delta(xz,z).
\end{equation*} 
\end{prop}
\begin{proof}
By Theorem~\ref{T:Diag},
\[
\Delta(z,xz)=\frac{f(z)}{g(z)-xz}
=\sum_{k\geq 0} x^k z^k \frac{f(z)}{g(z)^{k+1}}
=\sum_{k\geq 0} x^k C_k(z),
\]
where $C_k(z)$ is the generating function of the $k$th column of $D$. On the other hand, 
\[
\Delta(xz,z)=\sum_{k\geq 0} \Delta_k(z)\,(xz)^k=\sum_{k\geq 0} x^k\,z^k\Delta_k(z).
\]
Therefore, $\Delta(z,xz)=\Delta(xz,z)$ holds
if and only if $C_k(z)=z^k\Delta_k(z)$ for all  $k\geq 0$. By Proposition~\ref{prop:pal-columns-diagonals}, this is equivalent to $D$ being palindromic.
\end{proof}

It is well known that a Riordan matrix is uniquely determined by its first two columns. In the palindromic case, Proposition~\ref{prop:pal-columns-diagonals} shows that it suffices to impose
\[
C_0(x)=\Delta_0(x)
\qquad\text{and}\qquad
C_1(x)=x\,\Delta_1(x).
\]
Indeed, since
\[
C_0(x)=\frac{f(x)}{g(x)}
\qquad\text{and}\qquad
C_1(x)=\frac{x\,f(x)}{g(x)^2},
\]
these two series uniquely determine
\[
g(x)=\frac{x\,C_0(x)}{C_1(x)}
\qquad\text{and}\qquad
f(x)=g(x)\,C_0(x).
\]

\begin{thm}\label{T:Palin}
Let $D=T(f\mid g)$ be a Riordan matrix, where  $f=\sum_{n\geq0}f_nx^n$ and $g=\sum_{n\geq0}g_nx^n$, with $f_0\neq 0$ and $g_0\neq 0$. Then $D$ is a palindromic Riordan matrix if and only if 
\[f=\frac{f_0^2}{f_0-f_1x} \quad \text{and} \quad g=\frac{f_0(g_0-x)}{f_0-f_1x}.\] 
Equivalently,
\[
D=T\!\left(\frac{f_0^2}{f_0-f_1x}\ \Big|\ \frac{f_0(g_0-x)}{f_0-f_1x}\right).
\]
Consequently, for $n\geq 1$,\[f_n=\frac{f_1^n}{f_0^{n-1}}=f_1\left(\frac{f_1}{f_0}\right)^{n-1}\quad \text{and} \quad g_n=\left(\frac{f_1}{f_0}\right)^{n-1}\left(\frac{f_1g_0}{f_0}-1\right)=g_1\left(\frac{f_1}{f_0}\right)^{n-1}.\]
\end{thm}

\begin{proof}
The identities $C_0(x)=\Delta_0(x)$ and $C_1(x)=x\Delta_1(x)$ are equivalent to 
\[
\frac{f}{g}=\frac{f_0}{g_0-x}
\]
and
\[
\frac{xf}{g^2}=\frac{x}{g_0-x}\left(f_1-g_1\frac{f_0}{g_0-x}\right),
\]
respectively. Note that $d_{1,0}=d_{1,1}$ is equivalent to $f_0=f_1g_0-f_0g_1$. Using these relations and simplifying, we obtain 
\[
f=\frac{f_0^2}{f_0-f_1x} \qquad \text{and}\qquad g=\frac{f_0(g_0-x)}{f_0-f_1x}.
\]
Finally, since $f_0\neq0$ and $g_0\neq0$,  expanding these rational functions into power series gives,
for all $n\geq 1$,
\[
f_n=\frac{f_1^n}{f_0^{n-1}}  \quad  \text{and}\quad g_n=\left(\frac{f_1}{f_0} \right)^{n-1}\left(\frac{f_1g_0}{f_0}-1\right)=g_1\left(\frac{f_1}{f_0}\right)^{n-1}.
\]
\end{proof}

The above result is related to those obtained by Petrullo, \cite{p. Petrullo},  where the author also addresses the topic of palindromic Riordan matrices.

\begin{rmk}\label{r:parametros} Note that Theorem \ref{T:Palin} allows us to construct the set of all palindromic Riordan matrices in terms of three parameters: $f_0\neq0$, $g_0\neq0$, and $f_1$ arbitrary.
\end{rmk}

\begin{ejem}{\bf The Pascal's triangle.}
In this case, Pascal's triangle is obtained by taking $f_0=1$, $g_0=1$, and $f_1=0$. Indeed,
Theorem~\ref{T:Palin} gives $f(x)=1$ and $g(x)=1-x$, and therefore $
D=T(1\mid 1-x)$, so $d_{n,k}=\binom{n}{k}$.
\end{ejem}

\begin{ejem}{\bf Palindromic Toeplitz matrices.}
In this case, palindromic Toeplitz matrices are obtained by taking  $f_0=f_1\neq0$ and $g_0=1$. Therefore, 
\[
  \begin{pmatrix}
    f_0 \\
    f_0 & f_0&&&O \\
    f_0 & f_0  & f_0 \\
    f_0 & f_0  & f_0 & f_0\\
    f_0 & f_0  & f_0 & f_0 & f_0\\
    \vdots & \vdots &\vdots &\vdots &\vdots & \ddots \\
  \end{pmatrix}.
\]
 
\end{ejem}

\begin{ejem}{\bf Multiples of Pascal's triangle.}
These matrices can be obtained for $g_0=1$,  $f_1=0$, and any $f_0\neq0$,  
\[
  \begin{pmatrix}
    f_0 \\
    f_0 & f_0 &&&O\\
    f_0 & 2f_0  & f_0 \\
    f_0 & 3f_0  & 3f_0 & f_0\\
    f_0 & 4f_0  & 6f_0 & 4f_0 & f_0\\
    \vdots & \vdots &\vdots &\vdots &\vdots & \ddots \\
  \end{pmatrix}.
\]
\end{ejem}

\begin{ejem}{\bf Delannoy's triangle.}
It is well known that the Delannoy triangle, whose first rows are
\[
D=\left(
\begin{array}{ccccccc}
1 \\
1 & 1&&&&O \\
1 & 3 & 1 \\
1 & 5 & 5 & 1 \\
1 & 7 & 13 & 7 & 1 \\
1 & 9 & 25 & 25 & 9 & 1 \\
1 & 11 & 41 & 63 & 41 & 11 & 1
\end{array}
\right),
\]
is a palindromic Riordan matrix. If $D=(d_{n,k})$, then $d_{n,k}$ is the number of lattice paths from $(0,0)$ to $(n,k)$ using steps $(1,0)$, $(0,1)$, and $(1,1)$. As a Riordan matrix
\[
D=\left(\frac{1}{1-x},\frac{x(1+x)}{1-x}\right)=T\left(\frac{1}{1+x}\Big|\frac{1-x}{1+x}\right).
\]
Consequently, $f_0=g_0=1$ and $f_1=-1$.  The Delannoy matrix and its generalizations have been extensively studied; see, for example,
\cite{CheonKimShapiro, He-Yang, RamSir}.

 \end{ejem}

 \begin{ejem} {\bf Palindromic Riordan involutions and palindromic pseudoinvolutions.} 
 From Remark \ref{r:parametros},  any nontrivial palindromic involution satisfies $g_0=-1$, $f_0=\pm1$, and $f_1=a$. The first three rows of such a matrix should be
\[
T\left(\frac{1}{f_0-ax}\Big|\frac{-f_0(1+x)}{f_0-ax}\right)=
\begin{pmatrix}
    -\frac{1}{f_0} & & \\
    \frac{1}{f_0} & \frac{1}{f_0} & \\
    -\frac{1}{f_0} & -\frac{2}{f_0} -\frac{a}{f_0^2}&-\frac{1}{f_0}   
\end{pmatrix}.
\]
However, these matrices arise as the $3\times 3$ truncation of a Riordan involution if and only if $a=0$. Thus, the only palindromic Riordan involutions are $T(f_0\mid -1-x)$ with $f_0=\pm1$. Consequently, the only palindromic pseudo-involutions are Pascal's triangle
$P=T(1\mid 1-x)$ and its opposite $-P=T(-1\mid 1-x)$.
 \end{ejem}

 Recall that a \emph{polynomial sequence} is a family $\{p_n(x)\}_{n\ge 0}$ with $\deg p_n= n$ for each $n$. A direct consequence of Theorem~\ref{T:Palin} and Theorem~5 in \cite{poly} is the following.

\begin{cor}
Let $f_0,g_0,f_1\in\K$ with $f_0\neq 0$ and $g_0\neq 0$, and suppose that $g_1=(f_1g_0-f_0)/f_0$. 
Then a polynomial sequence $\{p_n(x)\}_{n\ge 0}$ is palindromic and of Riordan type if and only if
\[
p_n(x)=\frac{1}{g_0}\left[(x-g_1)p_{n-1}(x)+\left(\frac{f_1}{f_0}\right)^{n-1}p_0(x)
- g_1\frac{f_1}{f_0}\sum_{k=0}^{n-3}\left(\frac{f_1}{f_0}\right)^k p_{n-2-k}(x)\right]
\]
for all $n\ge 1$, with $p_0(x)=\frac{f_0}{g_0}$.
\end{cor}

\begin{cor}
Let $D=T(f\mid g)$ be a palindromic Riordan matrix. For $n\geq 0$ define the row polynomials
$p_n(t)=\sum_{k=0}^{n} d_{n,k}t^k$. Then
\[
\sum_{n\geq 0} p_n(t)\,z^n=\frac{f(z)}{g(z)-tz}
=\frac{f_0}{\,g_0-(1+t)z+\frac{f_1}{f_0}\,t z^2\,}.
\]
Consequently, $p_0(t)=\frac{f_0}{g_0}$, $p_1(t)=\frac{f_0}{g_0^2}(1+t)$, and for $n\geq 2$,
\[
g_0\,p_n(t)=(1+t)p_{n-1}(t)-\frac{f_1}{f_0}\,t\,p_{n-2}(t).
\]
\end{cor}
\begin{proof}
By (\ref{TFRA}), we have 
  \[
\sum_{n\geq 0} p_n(t)x^n
=
\frac{f(x)}{g(x)}\cdot \frac{1}{1-t\frac{x}{g(x)}}
=
\frac{f(x)}{g(x)-tx}.
\] If $D$ is palindromic, then Theorem~\ref{T:Palin} yields
\[
\frac{f(x)}{g(x)-tx}
=
\frac{f_0}{\,g_0-(1+t)x+\frac{f_1}{f_0}\,t x^2\,}.
\]
The initial values and the recurrence are followed by extracting the coefficients.
\end{proof}

\begin{cor}
Let $D=T(f\mid g)$ be a palindromic Riordan matrix. Then, for $0\leq k\leq n$,
\[
d_{n,k}
=\frac{f_0}{g_0^{\,n+1}}
\sum_{j=0}^{\min(k,n-k)}
\binom{k}{j}\binom{n-j}{k}\left(-\frac{f_1 g_0}{f_0}\right)^j.
\]
\end{cor}
\begin{proof}
By Theorem~\ref{T:Palin} and the Riordan coefficient formula,
\[
d_{n,k}=[x^{\,n-k}]\,\frac{f(x)}{g(x)^{k+1}}
=f_0\,[x^{\,n-k}]\,\frac{(f_0-f_1x)^k}{f_0^k\,(g_0-x)^{k+1}}.
\]
Expanding
\[
(f_0-f_1x)^k=\sum_{j=0}^{k}\binom{k}{j}f_0^{\,k-j}(-f_1)^j x^j
\qquad\text{and}\qquad
\frac{1}{(g_0-x)^{k+1}}
=\frac{1}{g_0^{k+1}}\sum_{m\geq 0}\binom{k+m}{k}\left(\frac{x}{g_0}\right)^m,
\]
and extracting $[x^{\,n-k}]$ gives the desired result.
\end{proof}

Notice that if $D$ is palindromic and $f_1=0$, then $f(x)=f_0$ and $g(x)=g_0-x$. In this case,
\[
d_{n,k}=\frac{f_0}{g_0^{\,n+1}}\binom{n}{k}\qquad(0\leq k\leq n),
\]
so $D$ is a scalar multiple of a rescaled Pascal matrix.

A previous version of Theorem \ref{T:Palin} was proposed by Hana Kim in September 2018, two months before her death. For this reason, we include her original proof.

\begin{prop}[Hana Kim, September 2018]\label{teo:Hanna}
If $\left(p_n(x)\right)_{n\ge0}$ is a sequence of palindromic
polynomials whose coefficient matrix is a Riordan matrix
$\left(d(z),h(z)\right)$, then
\begin{eqnarray*}
d(z)=\frac{d_0}{1-h_1z}\quad\text{and}\quad
h(z)=h_1z+\frac{h_2z^2}{1-h_1z},
\end{eqnarray*}
for $d_0,h_1,h_2\in\mathbb{C}$ with $d_0\ne0$ and $h_1\ne0$.
\end{prop}

\begin{proof} Let $p_n(x)=\sum_{k=0}^nd_{n,k}x^k$ so that
$(d(z),h(z))=[d_{n,k}]_{n,k\ge0}$ and $d_{n,k}=d_{n,n-k}$ for all $k$. Write $d(z)=\sum_{n\ge0}d_nz^n$ and $h(z)=\sum_{n\ge1}h_nz^n$. When $k=0$, it follows from $[z^n]d(z)=[z^n]d(z)h(z)^n$ that $d_n=d_0h_1^n$ for all $n$, or equivalently, $d(z)=d_0/(1-h_1 z)$. It can be easily shown that $[z^n]h(z)^{n-1}=(n-1)h_2h_1^{n-2}$. It then follows from $d_{n,1}=d_{n,n-1}$, or equivalently
$[z^n]d(z)h(z)=[z^n]d(z)h(z)^{n-1}$, that
\begin{eqnarray}\label{e:palin}
\sum_{k=0}^nd_0h_1^kh_{n-k}=d_0h_1^n+(n-1)d_0h_2h_1^{n-2}\quad(n\ge2)
\end{eqnarray}
where $h_0:=0$.  If $n=3$, we have $h_3=h_2h_1$. Suppose that $h_n=h_2h_1^{n-1}$. Then the left-hand side of \eqref{e:palin} (with $n$ replaced by $n+1$) is
\begin{align*}
\sum_{k=0}^{n+1}d_0h_1^kh_{n-k+1}&=d_0h_{n+1}+d_0h_1^{n+1}+d_0h_2h_1^{n-1}+\sum_{k=1}^{n-2}d_0h_2h_1^kh_1^{n-k-1}\\
&=d_0h_{n+1}+d_0h_1^{n+1}+(n-1)d_0h_2h_1^{n-1}
\end{align*}
while the right-hand side is $d_0h_1^{n+1}+nd_0h_2h_1^{n-1}$.
Comparing both sides gives $h_{n+1}=h_2h_1^{n-1}$, and therefore $h(z)=h_1z+h_2z^2/(1-h_1z)$.
\end{proof}

\begin{rmk}
Recall that $\frac{f}{g}=d$ and $\frac{x}{g}=h$. In particular, the equalities
\[
f_0=f_1g_0-f_0g_1
\qquad\text{and}\qquad
d_1=d_0h_1
\]
are equivalent. Both express the condition $d_{1,0}=d_{1,1}$, that is, $p_1(x)$ is palindromic.
\end{rmk}

\section{A combinatorial interpretation}

Next, we give a combinatorial interpretation of palindromic Riordan arrays, based on a lattice-path model. This approach is different from that of \cite{p. Petrullo}, which develops a broader framework encompassing other
generalizations of Riordan arrays. 

For integers $n,m\ge 0$, let $\D(n,m)$ denote the set of lattice paths in $\N\times\N$ from $(0,0)$ to $(n,m)$ using only the steps $H=(1,0)$ (horizontal), $V=(0,1)$ (vertical), and $D=(1,1)$ (diagonal). We call these \emph{Delannoy paths}. Figure~\ref{fig1} shows all paths in $\D(2,2)$.

Let $\Gamma_1,\Gamma_2\in \D(n,m)$.  We say that $\Gamma_1$ is \emph{equivalent} to $\Gamma_2$, and write $\Gamma_1\sim \Gamma_2$,  if $\Gamma_2$ is obtained from $\Gamma_1$
by finitely many local swaps of adjacent steps, replacing $HV$ by $VH$ (or vice versa). For example, the paths $\Gamma_1=VHDVHV$ and $\Gamma_2=HVDHVV$ are equivalent; see Figure~\ref{fig2}.

\begin{figure}[ht!]
\begin{tikzpicture}[scale=0.85,>=Stealth,line cap=round,line join=round]

\tikzset{
  gridline/.style={very thin,gray!55},
  pathseg/.style={thick,->},
  vtx/.style={circle,fill=black,inner sep=1.4pt}
}

\newcommand{\DelannoyGrid}[2]{%
  \draw[gridline] (0,0) grid (#1,#2);
  \draw[gridline] (0,0) rectangle (#1,#2);
}

\newcommand{\DelannoyPath}[1]{%
  \foreach \p [count=\i] in {#1} {%
    \ifnum\i=1
      \xdef\prev{\p}%
      \node[vtx] at \p {};
    \else
      \draw[pathseg] \prev -- \p;
      \node[vtx] at \p {};
      \xdef\prev{\p}%
    \fi
  }%
}

\begin{scope}[shift={(0,0)}]
  \DelannoyGrid{3}{4}
  \DelannoyPath{(0,0),(0,1),(1,1),(2,2),(2,3),(3,3),(3,4)}
  \node[anchor=west] at (0,4.35) {$\Gamma_1$};
\end{scope}

\draw[->,thick] (3.6,2) -- (4.6,2);

\begin{scope}[shift={(5.2,0)}]
  \DelannoyGrid{3}{4}
  \DelannoyPath{(0,0),(0,1),(1,1),(2,2),(3,2),(3,3),(3,4)}
\end{scope}

\draw[->,thick] (8.8,2) -- (9.8,2);

\begin{scope}[shift={(10.4,0)}]
  \DelannoyGrid{3}{4}
  \DelannoyPath{(0,0),(1,0),(1,1),(2,2),(3,2),(3,3),(3,4)}
  \node[anchor=west] at (0,4.35) {$\Gamma_2$};
\end{scope}

\end{tikzpicture}
\caption{Equivalent Delannoy paths.} \label{fig2}
\end{figure}
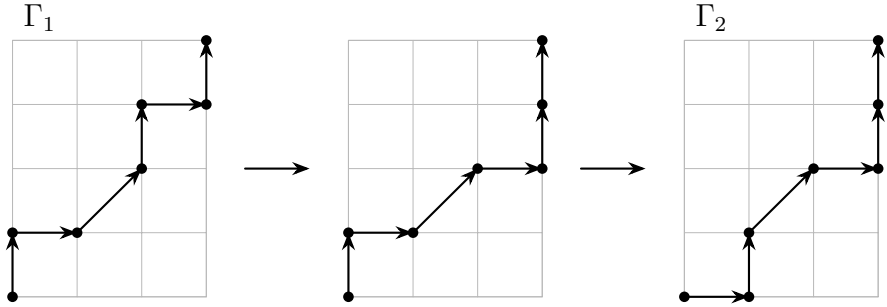
   
For fixed nonnegative integers $n$ and $m$, it is clear that $\sim$ is an equivalence relation on the set
$\D(n,m)$ of Delannoy paths. Given $\Gamma\in \D(n,m)$, we write $[\Gamma]$ for its equivalence class, namely
\[
[\Gamma]:=\{\Gamma'\in \D(n,m): \Gamma\sim \Gamma'\}.
\]
The quotient set of $\D(n,m)$ by $\sim$ is
\[
\overline{\D}(n,m):=\D(n,m)/\sim=\{[\Gamma]:\Gamma\in \D(n,m)\}.
\]
For example,
\[
\overline{\D}(2,2)=\{[\Gamma_1],[\Gamma_7],[\Gamma_8],[\Gamma_{10}],[\Gamma_{12}],[\Gamma_{13}]\},
\]
where $[\Gamma_1]=\{\Gamma_1,\Gamma_2,\Gamma_3,\Gamma_4,\Gamma_5,\Gamma_6\}$,
$[\Gamma_7]=\{\Gamma_7\}$, $[\Gamma_8]=\{\Gamma_8,\Gamma_9\}$,  $[\Gamma_{10}]=\{\Gamma_{10},\Gamma_{11}\}$,  $[\Gamma_{12}]=\{\Gamma_{12}\}$, $\Gamma_{13}]=\{\Gamma_{13}\}$; see Figure~\ref{fig1}.

\begin{figure}[ht!]
\centering
\begin{tikzpicture}[scale=0.90,>=Stealth,line cap=round,line join=round]

\tikzset{
  gridline/.style={very thin,gray!55},
  pathseg/.style={thick,->},
  vtx/.style={circle,fill=black,inner sep=1.4pt},
  pathlabel/.style={font=\large},
  panel/.style={inner sep=0pt,outer sep=0pt,draw=none},
  groupbox/.style={draw=green!55!black,line width=1.1pt,inner sep=8pt}
}

\newcommand{\DelannoyGrid}[2]{%
  \draw[gridline] (0,0) grid (#1,#2);
  \draw[gridline] (0,0) rectangle (#1,#2);
}
\newcommand{\DelannoyPath}[1]{%
  \foreach \p [count=\i] in {#1} {%
    \ifnum\i=1
      \xdef\prev{\p}%
      \node[vtx] at \p {};
    \else
      \draw[pathseg] \prev -- \p;
      \node[vtx] at \p {};
      \xdef\prev{\p}%
    \fi
  }%
}

\def\ytop{3.7}
\def\ybot{0}
\def\dx{3.0}

\begin{scope}[shift={(0*\dx,\ytop)}]
  \DelannoyGrid{2}{2}
  \DelannoyPath{(0,0),(0,1),(0,2),(1,2),(2,2)}
  \node[pathlabel,anchor=west] (L1) at (0,2.45) {$\Gamma_1$};
  \node[panel,fit={(0,0) (2,2) (L1)}] (P1) {};
\end{scope}

\begin{scope}[shift={(1*\dx,\ytop)}]
  \DelannoyGrid{2}{2}
  \DelannoyPath{(0,0),(0,1),(1,1),(1,2),(2,2)}
  \node[pathlabel,anchor=west] (L2) at (0,2.45) {$\Gamma_2$};
  \node[panel,fit={(0,0) (2,2) (L2)}] (P2) {};
\end{scope}

\begin{scope}[shift={(2*\dx,\ytop)}]
  \DelannoyGrid{2}{2}
  \DelannoyPath{(0,0),(0,1),(1,1),(2,1),(2,2)}
  \node[pathlabel,anchor=west] (L3) at (0,2.45) {$\Gamma_3$};
  \node[panel,fit={(0,0) (2,2) (L3)}] (P3) {};
\end{scope}

\begin{scope}[shift={(3*\dx,\ytop)}]
  \DelannoyGrid{2}{2}
  \DelannoyPath{(0,0),(1,0),(1,1),(2,1),(2,2)}
  \node[pathlabel,anchor=west] (L4) at (0,2.45) {$\Gamma_4$};
  \node[panel,fit={(0,0) (2,2) (L4)}] (P4) {};
\end{scope}

\begin{scope}[shift={(4*\dx,\ytop)}]
  \DelannoyGrid{2}{2}
  \DelannoyPath{(0,0),(1,0),(1,1),(1,2),(2,2)}
  \node[pathlabel,anchor=west] (L5) at (0,2.45) {$\Gamma_5$};
  \node[panel,fit={(0,0) (2,2) (L5)}] (P5) {};
\end{scope}

\begin{scope}[shift={(5*\dx,\ytop)}]
  \DelannoyGrid{2}{2}
  \DelannoyPath{(0,0),(1,0),(2,0),(2,1),(2,2)}
  \node[pathlabel,anchor=west] (L6) at (0,2.45) {$\Gamma_6$};
  \node[panel,fit={(0,0) (2,2) (L6)}] (P6) {};
\end{scope}

\begin{scope}[shift={(0*\dx,\ybot)}]
  \DelannoyGrid{2}{2}
  \DelannoyPath{(0,0),(1,1),(2,2)}
  \node[pathlabel,anchor=west] (L7) at (0,2.45) {$\Gamma_7$};
  \node[panel,fit={(0,0) (2,2) (L7)}] (P7) {};
\end{scope}

\begin{scope}[shift={(1*\dx,\ybot)}]
  \DelannoyGrid{2}{2}
  \DelannoyPath{(0,0),(0,1),(1,1),(2,2)}
  \node[pathlabel,anchor=west] (L8) at (0,2.45) {$\Gamma_8$};
  \node[panel,fit={(0,0) (2,2) (L8)}] (P8) {};
\end{scope}

\begin{scope}[shift={(2*\dx,\ybot)}]
  \DelannoyGrid{2}{2}
  \DelannoyPath{(0,0),(1,0),(1,1),(2,2)}
  \node[pathlabel,anchor=west] (L9) at (0,2.45) {$\Gamma_9$};
  \node[panel,fit={(0,0) (2,2) (L9)}] (P9) {};
\end{scope}

\begin{scope}[shift={(3*\dx,\ybot)}]
  \DelannoyGrid{2}{2}
  \DelannoyPath{(0,0),(1,1),(1,2),(2,2)}
  \node[pathlabel,anchor=west] (L10) at (0,2.45) {$\Gamma_{10}$};
  \node[panel,fit={(0,0) (2,2) (L10)}] (P10) {};
\end{scope}

\begin{scope}[shift={(4*\dx,\ybot)}]
  \DelannoyGrid{2}{2}
  \DelannoyPath{(0,0),(1,1),(2,1),(2,2)}
  \node[pathlabel,anchor=west] (L11) at (0,2.45) {$\Gamma_{11}$};
  \node[panel,fit={(0,0) (2,2) (L11)}] (P11) {};
\end{scope}

\begin{scope}[shift={(5*\dx,\ybot)}]
  \DelannoyGrid{2}{2}
  \DelannoyPath{(0,0),(0,1),(1,2),(2,2)}
  \node[pathlabel,anchor=west] (L12) at (0,2.45) {$\Gamma_{12}$};
  \node[panel,fit={(0,0) (2,2) (L12)}] (P12) {};
\end{scope}

\begin{scope}[shift={(2.5*\dx,\ybot-3.8)}]
  \DelannoyGrid{2}{2}
  \DelannoyPath{(0,0),(1,0),(2,1),(2,2)}
  \node[pathlabel,anchor=west] (L13) at (0,2.45) {$\Gamma_{13}$};
  \node[panel,fit={(0,0) (2,2) (L13)}] (P13) {};
\end{scope}

\node[groupbox,fit=(P1) (P6)]   {};
\node[groupbox,fit=(P7)]       {};
\node[groupbox,fit=(P8) (P9)]  {};
\node[groupbox,fit=(P10) (P11)]{};
\node[groupbox,fit=(P12)]      {};
\node[groupbox,fit=(P13)]      {};

\end{tikzpicture}
\caption{The $13$ Delannoy paths in $\D(2,2)$, grouped into equivalence classes under $\sim$.}
\label{fig1}
\end{figure}
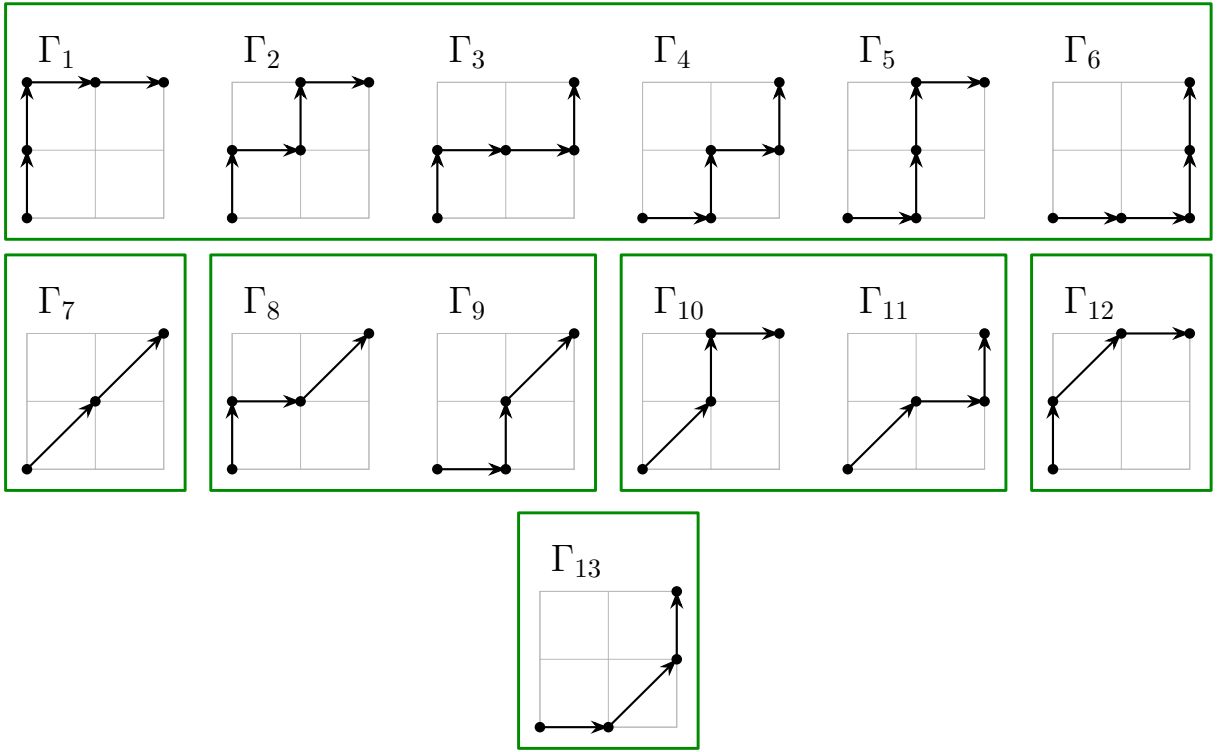

We denote by $\W_{a,b}(n,m)$ the set of Delannoy paths in $\D(n,m)$ in which each horizontal or vertical step
($H$ or $V$) has weight $a$, and each diagonal step $D$ has weight $b$. The \emph{weight} of a path $\Gamma\in \W_{a,b}(n,m)$ is the product of the weights of its steps. 

Note that a local exchange $HV\leftrightarrow VH$ does not change the multiset of steps, hence all paths in an
equivalence class $[\Gamma]$ have the same weight. We therefore define the weight of $[\Gamma]$ to be the weight
of any of its representatives (for instance, of $\Gamma$). We write $\wt(\Gamma)$ for the weight of a path $\Gamma$, and $\wt([\Gamma])$ for the (common) weight of any
representative of the equivalence class $[\Gamma]$.

Let $w_{a,b}(n,m)$ be the total weight of the quotient set $\overline{\D}(n,m)$, that is,
\[
w_{a,b}(n,m):=\sum_{[\Gamma]\in \overline{\D}(n,m)} \wt([\Gamma]).
\]
For example, from Figure~\ref{fig1} we read
\[
\wt([\Gamma_1])=a^4,\quad \wt([\Gamma_7])=b^2,\quad
\wt([\Gamma_8])=\wt([\Gamma_{10}])=\wt([\Gamma_{12}])=\wt([\Gamma_{13}])=a^2b,
\]
and thus $w_{a,b}(2,2)=a^4+4a^2b+b^2$.

By symmetry we have $w_{a,b}(n,m)=w_{a,b}(m,n)$ for all $n,m\ge 0$, and clearly
$w_{a,b}(n,0)=a^n=w_{a,b}(0,n)$ for all $n\ge 0$.

\begin{thm}\label{thm:wab}
For $n,m\ge 0$,
\[
w_{a,b}(n,m)=\sum_{k=0}^{\min\{n,m\}} \binom{n}{k}\binom{m}{k}\,a^{n+m-2k}b^k .
\]
\end{thm}

\begin{proof}
Partition $\D(n,m)$ according to the number $k$ of diagonal steps $D$. If a path has $k$ diagonals, then it has exactly $n-k$ horizontal steps and $m-k$ vertical steps, and hence
its weight is $a^{n+m-2k}b^k$. Therefore, it suffices to show that the quotient set $\overline{\D}(n,m)$ has
exactly $\binom{n}{k}\binom{m}{k}$ equivalence classes with $k$ diagonals.

Fix $k$. Let $\Cm_k(n,m)\subseteq \D(n,m)$ be the set of paths whose word has the canonical form
\[
\Gamma \;=\; V^{v_0}H^{h_0}\,D\,V^{v_1}H^{h_1}\,D\cdots D\,V^{v_k}H^{h_k},
\]
where $v_i,h_i\ge 0$ and
\[
v_0+\cdots+v_k=m-k,\qquad h_0+\cdots+h_k=n-k.
\] 

We claim that every equivalence class in $\D(n,m)$ with $k$ diagonals contains a unique element of $\Cm_k(n,m)$.
Indeed, the $k$ diagonals split any path word into $k+1$ blocks consisting only of $H$'s and $V$'s.
The local move $HV\leftrightarrow VH$ permutes $H$'s and $V$'s within each block, without changing their counts.
By repeatedly exchanging adjacent occurrences of $HV$ into $VH$, each block is brought to the form
$V^{v_i}H^{h_i}$, producing an element of $\Cm_k(n,m)$. Conversely, once a block is of the form
$V^{v_i}H^{h_i}$ it contains no subword $HV$, so no further exchanges are possible; hence the representative is unique.

Therefore, the number of equivalence classes with $k$ diagonals equals $|\Cm_k(n,m)|$.
Counting $\Cm_k(n,m)$ reduces to counting the nonnegative solutions of the two composition equations above.
By the stars-and-bars argument,
\[
|\Cm_k(n,m)|
=\binom{(m-k)+k}{k}\binom{(n-k)+k}{k}
=\binom{m}{k}\binom{n}{k}.
\]
Summing over $k$ and multiplying by the common weight $a^{n+m-2k}b^k$ yields the stated formula.
\end{proof}

\begin{cor}\label{c2}
Let $\W_{a,b}:=[w_{a,b}(n,m)]_{n,m\geq 0}$ be the weighted Delannoy matrix. Then
\[
\W_{a,b}=P_a\,D_b\,P_a^{T},
\]
where $P_a=\big[\binom{n}{k}a^{\,n-k}\big]_{n,k\ge 0}$ is the generalized Pascal matrix and
$D_b=\mathrm{diag}(1,b,b^2,\ldots)$.
\end{cor}
\begin{proof}
By Theorem~\ref{thm:wab},
\[
w_{a,b}(n,m)=\sum_{k=0}^{\min\{n,m\}}\binom{n}{k}\binom{m}{k}\,a^{n+m-2k}b^k
=\sum_{k\ge 0}\binom{n}{k}a^{\,n-k}\,b^k\,\binom{m}{k}a^{\,m-k},
\]
where the sum is finite since $\binom{n}{k}\binom{m}{k}=0$ for $k>\min\{n,m\}$.
This is precisely the $(n,m)$-entry of $P_aD_bP_a^T$, and the result follows.
\end{proof}

\begin{prop}\label{teorec}
For any weights $a$ and $b$, the numbers $w_{a,b}(n,m)$ satisfy the recursion
\[
w_{a,b}(n,m)=a\,w_{a,b}(n-1,m)+a\,w_{a,b}(n,m-1)+(b-a^2)\,w_{a,b}(n-1,m-1),
\quad n,m\ge 1.
\]
\end{prop}
\begin{proof}
Consider the equivalence classes of weighted Delannoy paths from $(0,0)$ to $(n,m)$.
A representative path may end with $H$, $V$, or $D$.
Appending $H$ to a path ending at $(n-1,m)$ contributes $a\,w_{a,b}(n-1,m)$, and appending $V$ to a path ending at
$(n,m-1)$ contributes $a\,w_{a,b}(n,m-1)$.
Appending $D$ to a path ending at $(n-1,m-1)$ contributes $b\,w_{a,b}(n-1,m-1)$.

However, paths whose last two steps are $HV$ and $VH$ belong to the same equivalence class.
More precisely, for any path $\Gamma\in \D(n-1,m-1)$, the two extensions $\Gamma HV$ and $\Gamma VH$ are equivalent,
and both have weight $a^2\,\wt(\Gamma)$. Thus, the diagonal contribution must be corrected by subtracting
$a^2w_{a,b}(n-1,m-1)$ to avoid double counting. Putting these contributions together yields the desired result.
\end{proof}

Let $W_n^{(a,b)}(z)$ be the ordinary generating function of the sequence $\{w_{a,b}(n,m)\}_{m\ge 0}$, that is,
\[
W_n^{(a,b)}(z)=\sum_{m\ge 0} w_{a,b}(n,m)\,z^m.
\]
The next theorem gives a closed form for $W_n^{(a,b)}(z)$.

\begin{thm}\label{teo1}
For $n\ge 0$,
\[
W_n^{(a,b)}(z)=\frac{\bigl(a+(b-a^2)z\bigr)^n}{(1-az)^{n+1}}.
\]
\end{thm}
\begin{proof}
From Proposition~\ref{teorec}, for $n\ge 1$ we have
$$W_n^{(a,b)}(z)=aW_{n-1}^{(a,b)}(z) + azW_{n}^{(a,b)}(z) + (b-a^2)zW_{n-1}^{(a,b)}(z).$$
Thus
\begin{align*}
W_n^{(a,b)}(x)&= \left(\frac{a + (b-a^2)z}{1-az}\right)W_{n-1}^{(a,b)}(z)= \left(\frac{a + (b-a^2)z}{1-az}\right)^nW_{0}^{(a,b)}(z)\\
&=\left(\frac{a + (b-a^2)z}{1-az}\right)^n\frac{1}{1-az}.
\end{align*}
\end{proof}

Let $\Qm_{a,b}:=[q_{a,b}(n,k)]_{n, k\geq 0}$ be the infinite matrix defined by
\[
q_{a,b}(n,k):=
\begin{cases}
w_{a,b}(n-k,k), & \text{if } n\ge k,\\[2pt]
0, & \text{if } n<k.
\end{cases}
\]
We call $\Qm_{a,b}$ the \emph{palindromic $(a,b)$-Delannoy matrix}. The terminology is justified by the fact that
each row of $\Qm_{a,b}$ is a palindromic sequence. The first few rows are
\[
\Qm_{a,b}=
\begin{pmatrix}
 1 & 0 & 0 & 0 & 0 & 0 \\
 a & a & 0 & 0 & 0 & 0 \\
 a^2 & a^2+b & a^2 & 0 & 0 & 0 \\
 a^3 & a^3+2ab & a^3+2ab & a^3 & 0 & 0 \\
 a^4 & a^4+3a^2b & a^4+4a^2b+b^2 & a^4+3a^2b & a^4 & 0 \\
 a^5 & a^5+4a^3b & a^5+6a^3b+3ab^2 & a^5+6a^3b+3ab^2 & a^5+4a^3b & a^5 \\
 \vdots & \ & \vdots & \vdots & \ & \vdots
\end{pmatrix}.
\]

\begin{thm}
For any complex numbers $a,b$ with $a\neq 0$, the palindromic $(a,b)$-Delannoy matrix $\Qm_{a,b}$ is the Riordan array
\[
\Qm_{a,b}=\left(\frac{1}{1-az},\, az+\frac{bz^2}{1-az}\right).
\]
\end{thm}
\begin{proof}
The $(n,k)$-entry of the Riordan array $\left(\frac{1}{1-az},\, az+\frac{bz^2}{1-az}\right)$ is
\begin{align*}
[z^n]\frac{1}{1-az}\left(az + \frac{bz^2}{1-az} \right)^k&=[z^n]\frac{1}{1-az}z^k\left(\frac{a(1-az) + bz}{1-az} \right)^k\\
&=[z^{n-k}]\frac{(a + (b-a^2)z)^k}{(1-az)^{k+1}}=[z^{n-k}]W_{k}^{(a,b)}(z)\\
&=w_{a,b}(n,n-k)=w_{a,b}(n-k,n)=q_{a,b}(n,k).
\end{align*}
Therefore, the Riordan array coincides with $\Qm_{a,b}$.
\end{proof}

In \cite{p. Petrullo}, the author considers a combinatorial interpretation for the palindromic polynomials of generalized Riordan arrays using directed graphs.

\section{Acknowledgements} G.-S Cheon was supported by the National Research Foundation of Korea(NRF) Grant funded by the Korean Government (MSIT)(RS-2025-00573047). The Spanish Government Grant PID2024-156663NB-I00 partially supported the second and third authors.
 J.~L.~R.~was partially supported by Universidad Nacional de Colombia, Project No.\ 64041.


\begin{thebibliography}{99}

\normalsize

\bibitem{Barnabei} {M.~Barnabei, A.~Brini, and G.~Nicoletti.}
{\it Recursive matrices and umbral calculus.} {J.  Algebra}
\textbf{75} (1982), 546--573.

\bibitem{Barnabei2} {M.~Barnabei, A.~Brini, and G.~Nicoletti.}
{\it A general umbral calculus in infinitely many variables.} {Adv. Math.} \textbf{50} (1983), 49--93.

\bibitem{Brini} { A.~Brini.} 
{\it Higher dimensional recursive matrices and diagonal delta sets of series.} {J. Combin. Theory Ser. A} \textbf{36} (1984), 315--331.


\bibitem{Cheon-Jin} {G.-S.~Cheon and S.-T.~Jin.}
{\it Structural properties of Riordan matrices and extending the matrices.} {Linear Algebra Appl. } \textbf{435} (2011),  {2019--2032}.

\bibitem{CheonKimShapiro}
G.-S.~Cheon, H.~Kim, and L.~W.~Shapiro. \emph{A generalization of the Lucas polynomial sequence},
Discrete Appl. Math.  \textbf{157} (2009), 920--927.



\bibitem{Lie} {G.-S.~Cheon, A.~Luz\'on, M.~A.~Mor\'on, L.~F.~ Prieto-Mart\'inez, and M.~Song.}
{\it Finite and infinite dimensional Lie group structures on Riordan groups.} {Adv. Math. } \textbf{319} (2017), 522--566.

\bibitem{q-conos} {P. Chocano, A. Luz\'{o}n, M. A. Mor\'{o}n, L.P. Prieto-Martinez}
{\it Riordan patterns' quest within simplicial complexes.} {Submitted.}


\bibitem{Jennings}
{S.~A.~Jennings.} {\it Substitution groups of formal power series.} {Canad. J. Math.} {\textbf{6}} (1954), {325--340}.



\bibitem{2ways}
{A.~Luz\'{o}n.} {\it Iterative processes related to Riordan arrays: The reciprocation and the inversion of power series. } {Discrete Math. \textbf{310}} (2010), {3607--3618}.

\bibitem{Constr}
{A.~Luz\'on, D.~Merlini, M.~A.~Mor\'on, and R.~Sprugnoli.} {\it Identities induced by Riordan arrays. } {Linear Algebra Appl. }{\textbf{436}} { (2012),} {631--647}.


\bibitem{Complem}
{A.~Luz\'{o}n, D.~Merlini, M.~A.~Mor\'{o}n, and R.~Sprugnoli.} {\it
Complementary Riordan arrays. } {Discrete Appl. Math. \textbf{172}} {(2014),
} {75--87}.

\bibitem{finitas}
{A.~Luz\'{o}n, D.~Merlini, M.~A.~Mor\'{o}n, L.~F.~Prieto-Mart\'{i}nez, and R.~Sprugnoli.} {\it Some inverse limit approaches to the Riordan group.} Linear Algebra Appl. \textbf{491} (2016), 239--262.


\bibitem{teo}
{A.~Luz\'{o}n and M.A.~Mor\'{o}n.} {\it Ultrametrics, Banach's fixed point theorem and the Riordan group. } {Discrete Appl. Math. \textbf{156}} (2008),
{2620--2635}.


\bibitem{BanPas}
{A.~Luz\'{o}n and M.~A.~Mor\'{o}n.} {\it Riordan matrices in the reciprocation of quadratic polynomials. } {Linear Algebra Appl.} \textbf{430} (2009),  2254--2270.


\bibitem{poly}
{A.~Luz\'on and M.~A.~Mor\'on.} {\it Recurrence relations for
polynomial sequences via Riordan matrices. } {Linear Algebra Appl. \textbf{433}} {(2010)}, 1422--1446.


\bibitem{commutador} {A.~Luz\'on, M.~A.~Mor\'on, and L.~F.~Prieto-Mart\'inez} {\it Commutators and commutator subgroups of the Riordan group.} {Adv. Math. }{\textbf{428}}{ (2023), } {109164}.


\bibitem{Double}
{A.~Luz\'{o}n, M.~A.~Mor\'{o}n, and J.~L.~Ram\'{i}rez.} {\it  Double parameter recurrences for polynomials in bi-infinite Riordan matrices and some derived identities. } {Linear Algebra Appl.} \textbf{511} (2016), 237--258.


\bibitem{Merlini}
{D.~Merlini, D.~G.~Rogers, R.~Sprugnoli, and M.~C.~Verri.} {\it On some alternative characterizations of Riordan arrays. }{ Canadian J. Math.} {\textbf{49}(2)} (1997) {301--320}.


\bibitem{p. Petrullo} 
{P.~Petrullo.} {\it Palindromic Riordan arrays, classical orthogonal polynomials, and Catalan triangles. } {Linear Algebra Appl. }{\textbf{618}} {(2021),}
{158--182}.

\bibitem{RamSir} J.~L.~Ram\'irez and  V.~Sirvent. {\it A generalization of the $k$-bonacci sequence from Riordan arrays}.  Electron. J. Combin. \textbf{22}(1) (2015), P1.38, 1--20.


\bibitem{Rog78}
{D.~G.~Rogers.} {\it Pascal triangles, Catalan numbers and renewal arrays.} {Discrete Math. \textbf{22}} (1978), {301--310}.

\bibitem{Sha91} 
{L.~W.~Shapiro, S.~Getu, W.-J.~Woan, and L.C.~Woodson}. \emph{The Riordan group}. Discrete Appl. Math. \textbf{34} (1991), 229--239.
 
\bibitem{Spr2008}
R.~Sprugnoli. \emph{Negation of binomial coefficients}.
 {Discrete  Math.} \textbf{308} (2008), 5070--5077.


\bibitem{Spr94}
{R. Sprugnoli.} {\it Riordan arrays and combinatorial sums. }
{Discrete Math. 132} (1994) {267--290}.



\bibitem{VerdeStar85}
{L.~Verde-Star.} {\it Dual operators and Lagrange inversion in several variables. } {Adv. Math. \textbf{58}} (1985), {89--108}.


\bibitem{VerdeStar04}
{L.~Verde-Star.} {\it Groups of generalized Pascal matrices. }
 {Linear Algebra Appl. \textbf{382}} (2004), {179--194}.

\bibitem{He-Yang} {S-L.~Yang, S-N.~Zheng, S-P.~Yuan, and T-X.~He.} {\it Schr\"oder matrix as inverse of Delannoy matrix.
}{Linear Algebra Appl.} {\textbf{439}} {(2013),} {3605--3614.}

\bibitem{Wilf}
H.~S.~Wilf. \emph{generatingfunctionology},
3rd ed., A~K~Peters/CRC Press, Wellesley, MA, 2005.



\end{thebibliography}
\end{document}